%% file: main.tex
\documentclass[12pt]{article}
\usepackage[margin=1in]{geometry}
\usepackage{tikz}
\usepackage{amsmath,amssymb,amsthm}
\usepackage[hidelinks]{hyperref}

\DeclareMathOperator{\sgn}{sgn}
\DeclareMathOperator{\loops}{loops}
\DeclareMathOperator{\lng}{long}
\DeclareMathOperator{\cy}{cy}
\DeclareMathOperator{\wt}{wt}

\input{pictures}

\newtheorem{theorem}{Theorem}[section]
\newtheorem{proposition}{Proposition}[section]
\newtheorem{corollary}{Corollary}[section]
\newtheorem{conjecture}{Conjecture}[section]

\theoremstyle{definition}
\newtheorem{observation}{Observation}[section]
\newtheorem{definition}{Definition}[section]

\title{Complements of coalescing sets}
\author{Steve Butler\footnote{Iowa State University, Ames, IA 50011, USA, \texttt{\{butler,joeljef}@iastate.edu}\and 
Elena D'Avanzo\footnote{Carleton College, Northfield, MN 55057, USA, \texttt{davanzoe@carleton.edu}} \and 
Rachel Heikkinen\footnote{Augustana College, Rock Island, IL, 61201, USA, \texttt{rachelheikkinen19@augustana.edu}} \and 
Joel Jeffries\footnotemark[1] \and 
Alyssa Kruczek\footnote{Susquehanna University, Selinsgrove, PA 17870, USA, \texttt{a.kruczek0413@gmail.com}} \and
Harper Niergarth\footnote{University of Minnesota--Twin Cities, Minneapolis, MN 55455, USA, \texttt{nierg001@umn.edu}}}
\date{\empty}

\begin{document}

\maketitle

\begin{abstract}
We consider matrices of the form $qD+A$, with $D$ being the diagonal matrix of degrees, $A$ being the adjacency matrix, and $q$ a fixed value. Given a graph $H$ and $B\subseteq V(G)$, which we  call a coalescent pair $(H,B)$, we derive a formula for the characteristic polynomial where a copy of same rooted graph $G$ is attached by the root to \emph{each} vertex of $B$. Moreover, we establish  if $(H_1,B_1)$ and $(H_2,B_2)$ are two coalescent pairs which are cospectral for any possible rooted graph $G$, then $(H_1,V(H_1)\setminus B_1)$ and $(H_2,V(H_2)\setminus B_2)$ will also always be cospectral for any possible rooted graph $G$.
\end{abstract}

\section{Introduction}
Given a graph $G$, we will consider matrices of the form $L_q = qD+A$, where $D$ is the diagonal matrix of degrees, $A$ is the adjacency matrix, and $q$ is a fixed value. For various values of $q$, we have well-known matrices including the adjacency matrix ($q=0$), the signless Laplacian ($q=1$), and the Laplacian ($q=-1$; technically the negation of the Laplacian). Given a matrix, we can consider the problem of understanding the structure of the graph from the eigenvalues of the matrix. Each of the previously mentioned matrices have different strengths and weaknesses for understanding the structure of a graph \cite{haemers,butler}.

One simple graph operation that can be used to form large graphs from smaller graphs is \emph{coalescing}, where two graphs are merged into one by identifying a single vertex. More formally, we have the following.

\begin{definition}
A \emph{coalescent pair} $(H,B)$ consists of a graph $H$ with $B\subseteq V(H)$. The \emph{coalescing} of $(H,B)$ with a rooted graph $G$ with root $r$ is formed by taking $|B|$ copies of $G$ and, for each copy, identifying the root $r$ with a different vertex in $B$.

If the coalescent pairs $(H_1,B_1)$ and $(H_2,B_2)$ are cospectral for every rooted graph $G$ with respect to some matrix, then we say that they are \emph{coalescing cospectral} with respect to that matrix.
\end{definition}

In Section~\ref{sec:charpoly} we derive a formula for the characteristic polynomial when coalescing $(H,B)$ with a rooted graph $G$. From this we will derive necessary and sufficient conditions for $(H_1,B_1)$ and $(H_2,B_2)$ to be coalescing cospectral for some matrix. In Section~\ref{sec:complement}, we look at the relationships between $(H,B)$ and $(H,V(H)\setminus B)$ in regards to the characteristic polynomials under gluing. Using this we will establish the following main result of this paper.

\begin{theorem}\label{thm:main}
If  $(H_1,B_1)$ and $(H_2,B_2)$ are coalescing cospectral with respect to $L_q$, then $(H_1,V(H_1)\setminus B_1)$ and $(H_2,V(H_2)\setminus B_2)$ are also coalescing cospectral with respect to $L_q$.
\end{theorem}

The case $(H_1,\emptyset)$ and $(H_2,\emptyset)$ can be restated in the following way.

\begin{corollary}\label{cor:emptyglue}
If graphs $H_1$ and $H_2$ are cospectral for $L_q$, then the graphs resulting from attaching an arbitrary rooted graph $G$ to \emph{each} vertex of $H_1$ and each vertex of $H_2$ will also be cospectral.
\end{corollary}

In Section~\ref{sec:unions}, we explore some other operations that might be carried out for coalescing sets. Finally, we conclude the paper in Section~\ref{sec:examples} with some examples for various values of $q$ and some additional applications.

\section{Characteristic polynomial for coalescing}\label{sec:charpoly}
Given a graph $H$ and $S\subseteq V(H)$, we let $p_{H,S}(x)$ be the characteristic polynomial for the matrix that results from taking $L_q$ for the graph and deleting the rows/columns associated with the vertices $S$. When $S=\{v\}$, a single vertex, we will write this as $p_{H,v}$. We will let $p_H(x)=p_{H,\emptyset}(x)$ denote the characteristic polynomial of the $L_q$ matrix for the graph $H$. When $q\ne 0$, $p_{H,D}$ is \emph{not} the same as finding the characteristic polynomial for the graph $H$ with the vertices $S$ deleted as there will be differences in the diagonal terms.

For the adjacency matrix ($L_0$), Schwenk \cite{schwenk} established that the graph $\mathcal{G}$ resulting from coalescing $(H,\{v\})$ with a rooted graph $G$ with root vertex $r$ is
\begin{equation}\label{eq:simplecase}
p_{\mathcal{G}}(x)
=\big(p_{G,r}(x)\big)p_{H,\emptyset}(x)+\big(p_{G,\emptyset}(x)-xp_{G,r}(x)\big)p_{H,v}(x).
\end{equation}
This was later generalized to the Laplacian matrix ($L_{-1}$) and signless Laplacian matrix ($L_{1}$) by Guo, Li, and Shiu \cite{Guo} with the same formulation. We now further generalize this result.

\begin{theorem}\label{thm:charpoly}
Let $(H,B)$ be a coalescent pair, let $\mathcal{G}$ be the graph resulting from coalescing $(H,B)$ with the graph $G$ rooted at $r$, and let $T\subseteq V(H)$ satisfy $T\cap B=\emptyset$. Then for the matrix $L_q=qD+A$, we have
\begin{equation}\label{eq:general_charpoly}
p_{\mathcal{G},T}(x)=\sum_{k=0}^{|B|}\bigg(\big(p_{G,r}(x)\big)^{|B|-k}\big(p_{G,\emptyset}(x)-xp_{G,r}(x)\big)^{k}
\sum_{\substack{S\subseteq B\\|S|=k}}p_{H,S\cup T}(x)\bigg).
\end{equation}
Setting $T=\emptyset$ gives the characteristic polynomial for $p_{\mathcal{G}}(x)$.
\end{theorem}

We first recall some basics of cycle decompositions for computing the characteristic polynomial for undirected graphs (and more generally symmetric matrices with $0$-$1$ entries in the off-diagonal). Given a graph $\mathcal{G}$ on $n$ vertices, then
\[
p_{\mathcal{G}}(x)=\det(xI-L_q)=\sum_{\sigma\in\mathcal{S}_n}(-1)^{\sgn(\sigma)}m_{1,\sigma(1)}\cdots m_{n,\sigma(n)}
\]
where $\mathcal{S}_n$ is the set of all permutations $\sigma$ of $\{1,\ldots,n\}$, $\sgn(\sigma)$ is the sign of the permutation, and $m_{i,j}=(xI-L_q)_{i,j}$. Any term on the right side involving an entry of the matrix which is $0$ will vanish, and so any remaining terms can be interpreted in the original graph as collections of (directed) edges in the graph which form collections of ``cycles'' (cycles take three forms: (1) loop at a single vertex; (2) an edge between two vertices; (3) cycle of length three or more). 

In particular, we have the following:
\[
p_{\mathcal{G}}(x)=\sum_{C\in\mathcal{C}}\underbrace{x^{u(C)}2^{\lng(C)}(-1)^{\cy(C)}\prod_{v\in\loops(C)}\big(q\deg_{\mathcal{G}}(v)\big)}_{=\wt(C)}=\sum_{C\in\mathcal{C}}\wt(C).
\]
This latter sum runs over all possible ways to have vertex disjoint cycles in the graph (these are known as the cycle decompositions and also include the choice of using no cycles at all).  We have $u(C)$ is the number of unused vertices not involved in any cycle, $\lng(C)$ is the number of cycles of length three or greater, $\cy(C)$ is the number of cycles, and $\loops(C)$ are the loops of the cycle decomposition.

We are separating out the diagonal terms into two cases so that $x$ will be when a vertex is unused and $q\deg(v)$ will be when a vertex is a loop. If we hold the rest of the cycle decomposition fixed and only change whether a particular vertex has a loop, the combination of the non-used and looped cases gives $x-q\deg(v)$ times the weight of the remaining cycle decomposition (recall that adding the loop adds one more cycle which causes a sign change).

More information about using cycle decompositions for computing characteristic polynomials can be found in Brualdi and Ryser \cite{brualdi}. 

\begin{proof}[Proof of Theorem~\ref{thm:charpoly}]
We will proceed by induction. For the base case of $B=\emptyset$, if we start with $H$, there is nowhere to glue on a copy of $G$, and so $\mathcal{G}=H$. Also, we have \eqref{eq:general_charpoly} becomes $p_{\mathcal{G},T}(x)=p_{H,T}(x)$, establishing the base case.

The key for induction will be the following formula. Let $\mathcal{G}$ be the result of coalescing $(H,\{v\})$ with the graph $G$ rooted at $r$. Then
\[
p_{\mathcal{G},T}(x)=
\underbrace{p_{G,r}(x)p_{H,T}(x)}_{\text{Case I}}+
\underbrace{p_{G,\emptyset}(x)p_{H,T\cup\{v\}}(x)}_{\text{Case II}}-
\underbrace{xp_{G,r}(x)p_{H,T\cup\{v\}}(x)}_{\text{Case III}}.
\]
When computing cycle decompositions, we proceed as before, only we discard any cycle which uses a vertex in $T$ (since that portion of the matrix is deleted). Since $v$ will be a cut-point in $\mathcal{G}$, any cycle involving $v$ can be wholly contained in $G$ or wholly contained in $H-T$. This allows us to break the possibilities up into three cases:
\begin{itemize}
\item Cycle decompositions where $v$ is in a cycle contained in $H-T$. All such possible cycle decompositions consist of a cycle decomposition in $H-T$ with a cycle decomposition in $G-r$, which will be accounted for in Case~I above.
\item Cycle decompositions where $v$ is in a cycle contained in $G$. All such possible cycle decompositions consist of a cycle decomposition in $H-(T\cup\{v\})$ with a cycle decomposition in $G$, which will be accounted for in Case~II above.
\item Cycle decompositions where $v$ is \emph{not} involved in any cycle. All such possible  cycle decompositions consist of a cycle decomposition in $H-(T\cup\{v\})$ with a cycle decomposition in $G-r$ and then an unused vertex ($v$), which will be accounted for in Case~III above.
\end{itemize}

Now we look at how a single cycle decomposition $C$ of $\mathcal{G}$ not involving any vertex in $T$ contributes to both sides.
\begin{itemize}
\item If $C$ involves $v$ in an edge or longer cycle, then if the cycle is in $H$, the contribution from Case I on the right is the same as the left, while if the cycle is in $G$, the contribution from Case II on the right is the same as the left.
\item If $C$ does not involve $v$ in any cycle, then each term from the three cases on the right makes the same contribution as it does on the left, but with the negation happening, the net contribution becomes the same.
\item If $C$ involves $v$ in a loop, then the contribution on the left is $\wt(C')qd_{\mathcal{G}}(v)$, where $C'$ is the portion of the cycle decomposition of $C$ not involving $v$. The contribution from the first and second terms on the right are $\wt(C')qd_{H}(v)$ and $\wt(C')qd_{G}(r)$, respectively, while the third term will not make a contribution. Since $d_{\mathcal{G}}(v)=d_H(v)+d_G(r)$, equality of the contribution to the two sides follows.
\end{itemize}
Finally, every cycle decomposition arising from terms on the right have been accounted for, and since in all cases the contribution matches, equality is established.

We now assume that the result holds for all graphs and all subsets $B'$ with $|B'|\le i$, and consider the case $(H,B)$ where $B=B'\cup \{v\}$ with $|B'|=i$. Let $\mathcal{G}'$ denote the graph resulting from coalescing $(H,B')$ with the graph $G$ rooted at $r$, and $\mathcal{G}$ denote the graph resulting from coalescing $(H,B)$ with the graph $G$ rooted at $r$ (which can also be interpreted as further coalescing $\mathcal{G}'$ at vertex $v$ with $G$ at vertex $r$). We now have
\begin{align*} 
p_{\mathcal{G},T}(x)&=
\big(p_{G,r}(x)\big)p_{\mathcal{G}',T}(x){+}\big(p_{G,\emptyset}(x){-}xp_{G,r}(x)\big)p_{\mathcal{G}',T\cup\{v\}}(x)\\
&=
\big(p_{G,r}(x)\big)\bigg(\sum_{S\subseteq B'}p_{H,S\cup T}(x)\big(p_{G,r}(x)\big)^{|B'|{-}|S|}\big(p_{G,\emptyset}(x){-}xp_{G,r}(x)\big)^{|S|}\bigg)
\\&\phantom{~=}{+}
\big(p_{G,\emptyset}(x){-}xp_{G,r}(x)\big)\bigg(\sum_{\substack{S\cup\{v\}\\S\subseteq B'}}p_{H,S\cup T\cup\{v\}}(x)\big(p_{G,r}(x)\big)^{|B'|{-}|S|}\big(p_{G,\emptyset}(x){-}xp_{G,r}(x)\big)^{|S|}\bigg)\\
& =
\big(p_{G,r}(x)\big)\bigg(\sum_{\substack{S\subseteq B\\v\notin S}}p_{H,S\cup T}(x)\big(p_{G,r}(x)\big)^{|B|{-}1{-}|S|}\big(p_{G,\emptyset}(x){-}xp_{G,r}(x)\big)^{|S|}\bigg)
\\&\phantom{~=}{+}
\big(p_{G,\emptyset}(x){-}xp_{G,r}(x)\big)\bigg(\sum_{\substack{S\subseteq B\\v\in S}}p_{H,S\cup T}(x)\big(p_{G,r}(x)\big)^{|B|{-}|S|}\big(p_{G,\emptyset}(x){-}xp_{G,r}(x)\big)^{|S|{-}1}\bigg)\\
&=\sum_{S\subseteq B}p_{H,S\cup T}(x)\big(p_{G,r}(x)\big)^{|B|{-}|S|}\big(p_{G,\emptyset}(x){-}xp_{G,r}(x)\big)^{|S|}\\
&=\sum_{k=0}^{|B|}\bigg(\big(p_{G,r}(x)\big)^{|B|{-}k}\big(p_{G,\emptyset}(x){-}xp_{G,r}(x)\big)^{k}
\sum_{\substack{S\subseteq B\\|S|=k}}p_{H,S\cup T}(x)\bigg).
\end{align*}
In the first step we applied the formula from above of gluing at a single vertex. We then applied the induction hypotheses for $p_{\mathcal{G}',T}(x)$ and $p_{\mathcal{G}',T\cup\{v\}}(x)$. In going to the third line, we rewrote the second line in terms of $B$, noting $|B'|=|B|-1$, and in the second part, since we insist on $v\in S$, we need to also have $|S|$ replaced by $|S|-1$. Distributing the terms in and combining the sums then gives us the fourth line. Finally, combining subsets of $S$ by their size and pulling out common factors gives our final line, establishing our result.
\end{proof}

Examining the expression from Theorem~\ref{thm:charpoly}, we see that the contributions from $H$ in computing the characteristic polynomials for $p_\mathcal{G}(x)$ are
\[
f_{H,B,k}(x) = \sum_{\substack{S\subseteq B\\|S|=k}}p_{H,S}(x).
\]
If these polynomials match between two different coalescing pairs, then the results of coalescing will result in graphs which have the same characteristic polynomial (are cospectral). This is stated more formally in the following result.

\begin{theorem}\label{thm:necessary}
Let $(H_1,B_1)$ and $(H_2,B_2)$ be coalescing pairs. Then $f_{H_1,B_1,k}(x)=f_{H_2,B_2,k}(x)$ for all $k$ if and only if $(H_1,B_1)$ and $(H_2,B_2)$ are coalescing cospectral.
\end{theorem}
\begin{proof}
Using Theorem~\ref{thm:charpoly}, with $T=\emptyset$, we have
\begin{multline*}
p_{\mathcal{G}_1}(x)=
\sum_{k=0}^{|B|}\Big(\big(p_{G,r}(x)\big)^{|B|-k}\big(p_{G,\emptyset}(x)-xp_{G,r}(x)\big)^{k}
f_{H_1,B_1,k}(x)\Big)\\
=\sum_{k=0}^{|B|}\Big(\big(p_{G,r}(x)\big)^{|B|-k}\big(p_{G,\emptyset}(x)-xp_{G,r}(x)\big)^{k}
f_{H_2,B_2,k}(x)\Big)=p_{\mathcal{G}_2}(x),
\end{multline*}
establishing the forward direction.

For the backward direction, we consider the family of graphs $G = K_{1,\ell}$ (stars) with the root $r$ being the vertex of degree $\ell$. A straightforward computation gives
\[
p_{G,r}(x)=(x-q)^\ell\quad\text{and}\quad
p_{G,\emptyset}(x)-xp_{G,r}(x)=\ell(x-q)^{\ell-1}(q^2-qx-1).
\]
This then gives
\[
\big(p_{G,r}(x)\big)^{|B|{-}k}\big(p_{G,\emptyset}(x){-}xp_{G,r}(x)\big)^{k}{=} 
(x-q)^{\ell|B|}\big(\underbrace{\ell(q^2-qx-1)/(x-q)}_{=Y}\big)^k=(x-q)^{\ell|B|-1}Y^k.
\]
We now have
\begin{equation}\label{eq:transcend}
0=p_{\mathcal{G}_1}(x)-p_{\mathcal{G}_2(x)}=\sum_{k}(x-q)^{\ell|B|-1}\big(f_{H_1,B_1,k}(x)-f_{H_2,B_2,k}(x)\big)Y^k.
\end{equation}
Fix a value $x$ so that $x\ne q$ and $x\ne (q^2-1)/q$. We can now treat \eqref{eq:transcend} as a polynomial expression in $Y$, which is $0$ infinitely often (choosing $\ell=1,2,3,\ldots$). This can only happen if the polynomial is identically $0$, which means that the coefficient of $Y^k$ is $0$ for all $k$, which further implies that $f_{H_1,B_1,k}(x)-f_{H_2,B_2,k}(x)=0$ for all $k$ for our fixed value of $x$. Since there are infinitely many choices available for $x$, that in turn implies that these polynomials always agree, which is possible only if they are equal, establishing the result. 
\end{proof}

The proof of Theorem~\ref{thm:necessary} shows that if we can consistently coalesce rooted stars into two graphs and maintain cospectrality, then we can coalesce \emph{any} rooted graph and maintain cospectrality. For the adjacency matrix, this can be simplified even further by noting that for stars, the pendent vertices act as twins and so have simple eigenvectors (e.g.\ assigning $1$ to one leaf and $-1$ to another), for the remaining eigenvalues we can then do equitable partitions where we have each individual vertex of $H$ along with groupings of the leaves of every coalesced star. From there, a simple similarity relationship shows that the remaining eigenvalues are found by replacing the coalescing of the stars with the coalescing of a single edge with \emph{edge weight} $\sqrt\ell$ (where we glued $K_{1,\ell}$). We summarize this in the following observation.

\begin{observation}
For the adjacency matrix, we have $(H_1,B_1)$ and $(H_2,B_2)$ are cospectral coalescent for arbitrary rooted graphs if and only if $(H_1,B_1)$ and $(H_2,B_2)$ are cospectral coalescent for the family $G=K_2$ with arbitrary edge weight.
\end{observation}

This is similar to a result of Schwenk \cite{schwenk2}. We do not pursue this direction further here as our primary focus will be the characteristic polynomial and cycle decompositions. (The reason that this works well for the adjacency but not other matrices of the form $qD+A$ is because of what happens with the diagonal terms, in particular we would still add an edge of weight $\sqrt\ell$, but the degree terms on the diagonal would no longer agree with what is happening in the graph.)

\section{Coalescing on the complement}\label{sec:complement}
The goal of this section is to look at the relationships between the polynomial family $f_{H,B,k}(x)$ and the polynomial family $f_{H,V(H)\setminus B,k}(x)$. That is to say we are interested in examining the situation of coalescing on a set $B$ of vertices of $H$ and relating that to coalescing on the set $V(H)\setminus B$ of vertices of $H$.  The result of this section can be summarized as follows.

\begin{theorem}\label{thm:complement}
The polynomial family $f_{H,V(H)\setminus B,k}(x)$ can be determined from the polynomial family $f_{H,B,k}(x)$.
\end{theorem}

\begin{proof}
Since the individual characteristic polynomials are found by taking combinations of weights of cycle decompositions, we first describe the polynomials in terms of decompositions. We start by writing
\[
f_{H,B,k}(x) = 
\sum_{\substack{S\subseteq B\\|S|=k}}p_{H,S}(x) = 
\sum_{\substack{S\subseteq B\\|S|=k}}\sum_{\substack{C\in\mathcal{C}\\C\cap S = \emptyset}}\wt(C).
\]
For a fixed cycle decomposition $C$, the number of times that $\wt(C)$ shows up on the right will correspond with how many subsets $S$ can be chosen that avoids $C$. To determine this, it is useful to further refine cycle decompositions by how much they intersect with $B$ and with $V(G)\setminus B$. So, we introduce $w(s,t)$ defined for a graph $H$ as follows
\[
w(i,j)=\sum_{\substack{C\in\mathcal{C}\\|C\cap B| = i \\ |C\cap (V(H)\setminus B)|=j}}\wt(C).
\]
By convention, $w(i,j)=0$ if $i>|B|$, $j>|V(H)|-|B|$, $i<0$, or $j<0$.

Suppose that $C$ is a cycle decomposition that contributes to $w(i,j)$. Then the total number of times it will contribute to $f_{H,B,k}(x)$ will be $\binom{|B|-i}{k}$, where $|B|-i$ are the number of vertices of $B$ not in $C$, and we must choose $S$ as a subset of size $k$ from among them. Since this is true for every cycle decomposition that contributes to $w(i,j)$, and each cycle decomposition shows up in \emph{some} $w(i,j)$, we have
\[
f_{H,B,k}(x) = \sum_i\sum_j\binom{|B|-i}kw(i,j).
\]

We could alternatively write our polynomials as 
\[
f_{H,B,k}(x)=\sum_\ell c_{k,\ell}x^{|V(H)|-\ell}
\]
for appropriate constant coefficients $c_{k,\ell}$. To connect these, we see for any cycle which contributes to $w(i,j)$ that the cycle decomposition uses $i+j$ vertices among cycles and the remaining $|V(H)|-i-j$ vertices are unused. That means that $w(i,j)=\omega_{i,j}x^{|V(H)|-i-j}$ for some constant $\omega_{i,j}$. Putting this together we have
\[
\sum_\ell c_{k,\ell}x^{|V(H)|-\ell}=
\sum_i\sum_j\binom{|B|-i}k\omega_{i,j}x^{|V(H)|-i-j}=\sum_\ell\bigg(\sum_i\binom{|B|-i}k\omega_{i,\ell-i}\bigg)x^{|V(H)|-\ell}
\]
and so
\[
c_{k,\ell}=\sum_i\binom{|B|-i}k\omega_{i,\ell-i}.
\]

We now gather these equations into systems where we fix $\ell$ and let $k$ vary. We will consider two cases, namely when $\ell<|B|$ and $\ell\ge |B|$. When $\ell<|B|$, we have the following system of equations
\[
\begin{pmatrix}
c_{0,\ell}\\
c_{1,\ell-1}\\
\vdots\\
c_{\ell,0}
\end{pmatrix}=
\begin{pmatrix}
\binom{|B|-0}{0}&\binom{|B|-1}{0}&\cdots&\binom{|B|-\ell}{0}\\
\binom{|B|-0}{1}&\binom{|B|-1}{1}&\cdots&\binom{|B|-\ell}{1}\\
\vdots&\vdots&\ddots&\vdots\\
\binom{|B|-0}{\ell}&\binom{|B|-1}{\ell}&\cdots&\binom{|B|-\ell}{\ell}
\end{pmatrix}
\begin{pmatrix}
\omega_{0,\ell}\\
\omega_{1,\ell-1}\\
\vdots\\
\omega_{\ell,0}
\end{pmatrix}.
\]
We now claim that the matrix on the right side is invertible. Given that this is the case, this allows us to solve for $\omega_{i,j}$ with $i+j<|B|$ in terms of the $c_{i,j}$ by multiplying both sides by the inverse.

To verify our claim, we will show that the matrix has determinant $\pm1$. Start with the given matrix, and for each column but the last, subtract the next column from itself. This operation preserves the determinant, and using $\binom{a}{b}-\binom{a-1}{b}=\binom{a-1}{b-1}$ the result is the matrix
\[
\left(\begin{array}{cccc|c}
0&0&\cdots&0&\binom{|B|-\ell}{0}\\[5pt] \hline &&&&\\[-8pt]
\binom{|B|-1}{0}&\binom{|B|-2}{0}&\cdots&\binom{|B|-\ell}{0}&\binom{|B|-\ell}{1}\\
\binom{|B|-1}{1}&\binom{|B|-2}{1}&\cdots&\binom{|B|-\ell}{1}&\binom{|B|-\ell}{2}\\
\vdots&\vdots&\ddots&\vdots&\vdots\\
\binom{|B|-1}{\ell-1}&\binom{|B|-2}{\ell-1}&\cdots&\binom{|B|-\ell}{\ell-1}&\binom{|B|-\ell}{\ell}
\end{array}\right).
\]
The lower left sub-matrix is equivalent to a portion of the original matrix (in particular what comes from deleting the first column and last row). So we can carry out this operation repeatedly with the resulting lower left submatrix until we reduce to a matrix with $1$ on the anti-diagonal and $0$s above the anti-diagonal, which has determinant $\pm1$.

Now we consider the case when $\ell\ge |B|$. We note the polynomial family from which the coefficients come from only go up through $k=|B|$, and at the same time recall that $w_{i,j}=0$ for $i>|B|$ (since we cannot have more than $|B|$ vertices in a cycle decomposition $C$ intersecting $B$). So in this case we have
\[
\begin{pmatrix}
c_{0,\ell}\\
c_{1,\ell-1}\\
\vdots\\
c_{|B|,\ell-|B|}
\end{pmatrix}=
\begin{pmatrix}
\binom{|B|-0}{0}&\binom{|B|-1}{0}&\cdots&\binom{|B|-|B|}{0}\\
\binom{|B|-0}{1}&\binom{|B|-1}{1}&\cdots&\binom{|B|-|B|}{1}\\
\vdots&\vdots&\ddots&\vdots\\
\binom{|B|-0}{|B|}&\binom{|B|-1}{|B|}&\cdots&\binom{|B|-|B|}{|B|}
\end{pmatrix}
\begin{pmatrix}
\omega_{0,\ell}\\
\omega_{1,\ell-1}\\
\vdots\\
\omega_{|B|,\ell-|B|}
\end{pmatrix}.
\]
We can directly see that the matrix is invertible as the diagonal entries are $1$ and the entries below the diagonal are all $0$, and so the matrix has determinant $1$. Therefore, we are again able to solve for $\omega_{i,j}$ with $i+j\ge |B|$ in terms of the $c_{i,j}$.

So summarizing what we have done, we started with the polynomial family $f_{H,B,k}(x)$ and then rewrote that in terms of combinations of the $w(i,j)$. We then showed that given all of the coefficients of $f_{H,B,k}(x)$ that all of the $w(i,j)$ could be determined (we already knew their power of $x$ and then we determined the $\omega_{i,j}$ which was the scaling factor). To finish the argument, it now suffices to show how to compute the polynomial family $f_{H,V(H)\setminus B,k}(x)$ in terms of $w(i,j)$. Given the symmetric nature of the definition of $w(i,j)$ this is readily done as before, and we have
\[
f_{H,V(H)\setminus B,k}(x) = \sum_i\sum_j\binom{|V(H)|-|B|-j}kw(i,j).\qedhere
\]
\end{proof}

We are now ready to prove our main result.

\begin{proof}[Proof of Theorem~\ref{thm:main}]
By Theorem~\ref{thm:necessary} we have that $(H_1,B_1)$ and $(H_2,B_2)$ coalescing cospectral if and only if $f_{H_1,B_1,k}(x)=f_{H_2,B_2,k}(x)$ for all $k$. Now we can apply Theorem~\ref{thm:complement}  and conclude that $f_{H_1,V(H_1)\setminus B_1,k}(x)=f_{H_2,V(H_2)\setminus B_2,k}(x)$ for all $k$ (since they are both derived from the same polynomial family). We finally use Theorem~\ref{thm:necessary} again (in the other direction) and conclude that $(H_1,V(H_1)\setminus B_1)$ and $(H_2,V(H_2)\setminus B_2)$ are coalescing cospectral.
\end{proof}

\section{Unions of coalescing sets}\label{sec:unions}
Theorem~\ref{thm:main} shows that if we have $(H_1,B_1)$ and $(H_2,B_2)$ are coalescing cospectral, then we can find another coalescing cospectral pair by looking at the complements of the sets. A natural question arises if there are other set operations that we can perform that lead to other coalescing cospectral pairs. In this section we will demonstrate some of the challenges that can arise by considering unions.

To begin with, it is \emph{not} the case that the (disjoint) union of coalescing cospectral pairs become coalescing cospectral; this means that unions are not guaranteed to produce results which remain coalescing cospectral. An example of what can happen is given in Figure~\ref{fig:no_unions}.

\begin{figure}[!htb]
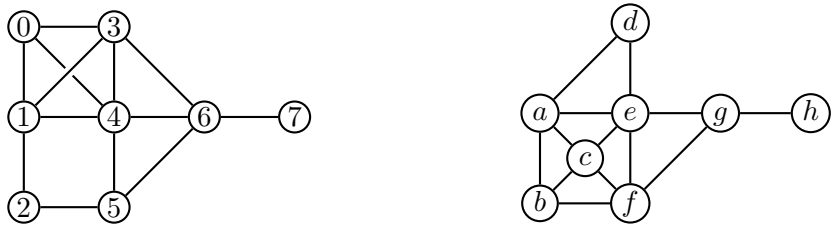

\centering
\FIGNOUNIONS
\caption{For $L_0=A$, we have $(H_1,\{1\})$ and $(H_2,\{a\})$ are coalescing cospectral; $(H_1,\{0,6\})$ and $(H_2,\{c,g\})$ are coalescing cospectral; but $(H_1,\{0,1,6\})$ and $(H_2,\{a,c,g\})$ are \emph{not} coalescing cospectral.}
\label{fig:no_unions}
\end{figure}

The example in Figure~\ref{fig:no_unions} is a demonstration that coalescing cospectral pairs do not have to remain as coalescing cospectral pairs after a coalescence has occurred somewhere else in the graph. This is also tied to the idea of simultaneously coalescing two different graphs onto two different coalescing sets. We have the following result in this direction.

\begin{theorem}\label{thm:twostep}
Fix a matrix $L_q$. If $(H_1,B_1)$ and $(H_2,B_2)$ are coalescing cospectral, $(H_1,\{v_1\})$ and $(H_2,\{v_2\})$ are coalescing cospectral, and $(H_1,B_1\cup\{v_1\})$ and $(H_2,B_2\cup\{v_2\})$ are coalescing cospectral, then the graphs formed by coalescing $G$ rooted at $r$ onto the sets $B_1$ and $B_2$ and also coalescing $\widehat{G}$ rooted at $\widehat{r}$ onto the vertices $v_1$ and $v_2$ will also be cospectral.
\end{theorem}

For the special case $B_i=V(H_i)\setminus\{v_i\}$ we have the following which is a strengthening of a result of Schwenk \cite{schwenk} by noting we could also glue arbitrary graphs into the ``non-root'' vertices.

\begin{corollary}
If two rooted graphs $H_1$ and $H_2$ are cospectral with respect to $L_q$ and the matrix of $H_1$ after deleting the row/column corresponding to $v_1$ is cospectral with the matrix of $H_2$ after deleting the row/column corresponding to $v_2$, then we can coalesce the same arbitrary graph onto $v_1$ and $v_2$ and a different arbitrary graph on all remaining vertices and the resulting graphs will be cospectral.
\end{corollary}

\begin{proof}[(Sketch of the) Proof of Theorem~\ref{thm:twostep}]
We begin by looking at what happens generally when we coalesce $G$ onto the vertices of $B$ and $G'$ onto the vertices of $B'$. Applying Theorem~\ref{thm:charpoly} we have
\[
p_{\mathcal{G}}(x)=\sum_{k,\ell}\big(\text{terms from $G$ and $\widehat{G}$}\big)\boxed{\,\sum_{\substack{S\subseteq B, |S|=k\\
T\subseteq B', |T|=\ell}}p_{H,S\cup T}(x)\,}\,.
\]
Denoting the boxed polynomial as $p_{H,B,B',k,\ell}(x)$ then by the same argument as Theorem~\ref{thm:necessary} we have that for two graphs they will be cospectral when we coalesce in all possible ways if and only if $p_{H_1,B_1,B'_1,k,\ell}(x)=p_{H_2,B_2,B'_2,k,\ell}(x)$ for all values of $k$ and $\ell$. (Similar statements hold for gluing into three or more parts.) On a side note, for general $B$ and $B'$, the number of conditions needed to be satisfied grows as the product of their sizes making it impractical to check for most cases.

Specifically, applying this to our case we must verify that for all $k,\ell$ that
\[
\sum_{\substack{S\subseteq B_1, |S|=k\\
T\subseteq \{v_1\}, |T|=\ell}}p_{H_1,S\cup T}(x)
=
\sum_{\substack{S\subseteq B_2, |S|=k\\
T\subseteq \{v_2\}, |T|=\ell}}p_{H_2,S\cup T}(x).
\]
We now check to see how to use our assumptions to verify that our collection of conditions are satisfied. We have several cases:
\begin{itemize}
\item For $\ell=0$ the equations reduce to showing $\sum p_{H_1,S}(x)=\sum p_{H_2,S}(x)$ which holds since $(H_1,B_1)$ and $(H_2,B_2)$ are coalescing cospectral.
\item For $k=0$ the equations reduce to showing $\sum p_{H_1,T}(x)=\sum p_{H_2,T}(x)$ which holds since $(H_1,\{v_1\})$ and $(H_2,\{v_2\})$ are coalescing cospectral.
\item For $\ell=1$ and $k>0$ the equations reduce to showing
\[
\sum_{\substack{S\subseteq B_1\\|S|=k}}p_{H_1,S\cup\{v_1\}}(x)=
\sum_{\substack{S\subseteq B_2\\|S|=k}}p_{H_2,S\cup\{v_2\}}(x).
\]
The key idea is to see that we can rewrite this as a combination and in particular this is equivalent to
\[
\sum_{\substack{S\subseteq B_1\cup\{v_1\}\\|S|=k+1}}p_{H_1,S}(x)
-\sum_{\substack{S\subseteq B_1\\ |S|=k+1}}p_{H_1,S}(x)=
\sum_{\substack{S\subseteq B_2\cup\{v_2\}\\ |S|=k+1}}p_{H_2,S}(x)
-\sum_{\substack{S\subseteq B_2\\ |S|=k+1}}p_{H_2,S}(x).
\]
(Find all the $k+1$ element subsets and then remove those missing $v_1$.) The first terms match since $(H_1,B_1\cup\{v_1\})$ and $(H_2,B_2\cup\{v_2\})$ are coalescing cospectral, while the second terms match since $(H_1,B_1)$ and $(H_2,B_2)$ are coalescing cospectral.
\end{itemize}
Since all cases are satisfied the result follows.
\end{proof}

While having one of the sets consist of a single vertex seems very limited, it is also best possible as the example shown in Figure~\ref{fig:2+2} demonstrates. In general, coalescing sets, at least in some cases, seem to be sensitive to coalescings happening elsewhere in the graph. There still remains much that is not known about coalescing sets and various graph operations.

\begin{figure}[!htb]
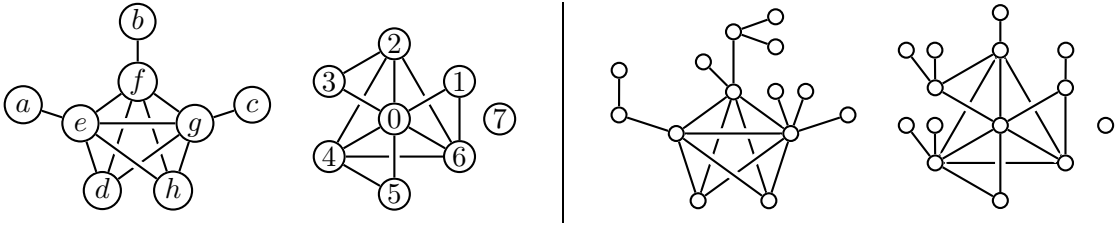

\centering
\FIGNOTWO
\caption{For $L_0$ (the adjacency matrix) we have  $(H_1,\{a,f\}$) and $(H_2,\{1,2\})$ are coalescing cospectral, $(H_1,\{b,g\}$) and $(H_2,\{3,4\})$ are coalescing cospectral, and $(H_1,\{a,b,f,g\}$) and $(H_2,\{1,2,3,4\})$ are coalescing cospectral. However, the graphs resulting from gluing in a $K_{1,1}$ into the first coalescing pairs and a $K_{1,2}$ into the second coalescing pairs (shown on the right) are \emph{not} cospectral.}
\label{fig:2+2}
\end{figure}

\section{Examples and additional remarks}\label{sec:examples}

One immediate question that arises is how general the phenomenon of coalescing cospectral pairs is for $L_q$. Corollary~\ref{cor:emptyglue} uses the special case when the coalescing set is empty (so that every cospectral pair has at least two such sets, empty and all; some graphs have only these). For many small graphs there are multiple coalescing sets available as illustrated in Figure~\ref{fig:0_example} ($L_0$, adjacency), Figure~\ref{fig:-1_example} ($L_{-1}$, Laplacian), Figure~\ref{fig:1_example} ($L_1$, signless Laplacian), and Figure~\ref{fig:1/2_example} ($L_{1/2}=\frac12D+A$). Throughout this section, we will adapt an abbreviated notation so, for example, $12{\,:\,}bd$ will mean that $(H_1,\{1,2\})$ is coalescing cospectral with $(H_2,\{b,d\})$ (where $H_1$ and $H_2$ can be determined from context by how vertices are labeled).

\begin{figure}[!htb]
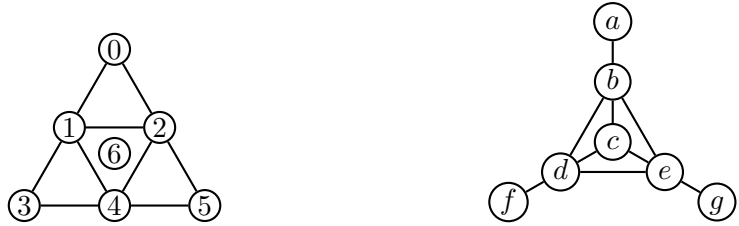

\centering
\FIGADJACENCY
\caption{Two graphs which have the following coalescing cospectral sets for $L_0=A$ (the adjacency) up to symmetry and taking complements: $\emptyset{\,:\,}\emptyset$, $1{\,:\,}b$, $12{\,:\,}bd$, $124{\,:\,}bde$.}
\label{fig:0_example}
\end{figure}

\begin{figure}[!htb]
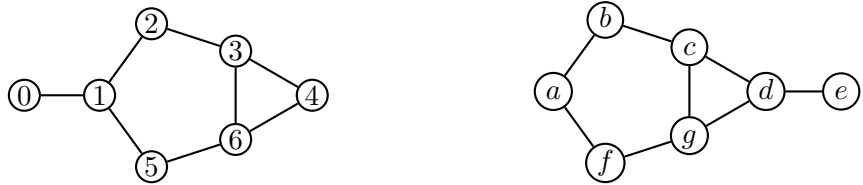

\centering
\FIGLAPLACIAN
\caption{Two graphs which have the following coalescing cospectral sets for $L_{-1}=-D+A$ (the Laplacian) up to symmetry and taking complements: $\emptyset{\,:\,}\emptyset$, $0{\,:\,}e$, $1{\,:\,}d$, $3{\,:\,}d$, $4{\,:\,}a$, $01{\,:\,}de$, $04{\,:\,}ae$, $14{\,:\,}ad$, $014{\,:\,}ade$.}
\label{fig:-1_example}
\end{figure}

\begin{figure}[!htb]
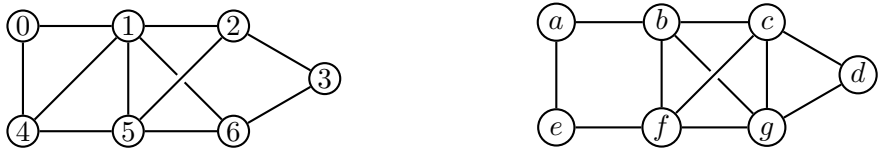

\centering
\FIGSIGNLESS
\caption{Two graphs which have the following coalescing cospectral sets for $L_1=D+A$ (the signless Laplacian) up to symmetry and taking complements: $\emptyset{\,:\,}\emptyset$, $0{\,:\,}d$, $3{\,:\,}a$, $5{\,:\,}f$, $03{\,:\,}ad$, $05{\,:\,}df$, $35{\,:\,}af$, $035{\,:\,}adf$.}
\label{fig:1_example}
\end{figure}

\begin{figure}[!htb]
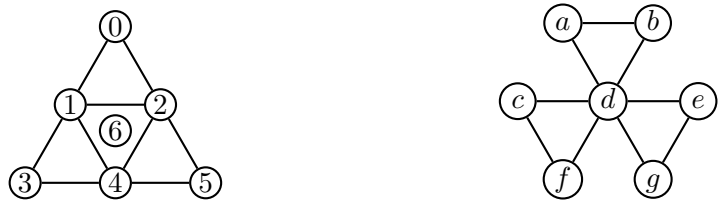

\centering
\FIGHALFEXAMPLE
\caption{Two graphs which have the following coalescing cospectral sets for $L_{1/2}=\frac12D+A$ up to symmetry and taking complements: $\emptyset{\,:\,}\emptyset$, $0{\,:\,}a$, $03{\,:\,}ac$, $035{\,:\,}ace$.}
\label{fig:1/2_example}
\end{figure}

\subsection*{An application}
This paper grew out of looking at examples of trees which were cospectral with non-trees (see \cite{faux}). In particular, for the signless, Laplacian the following was established.

\begin{theorem}[Butler et al.\ \cite{faux}]\label{thm:4k}
A tree $T$ can only be cospectral with a non-tree $G$ if and only if the number of vertices is $n=4k$.
\end{theorem}

The question then arises whether there is always a tree/non-tree pair when $n=4k$. This can now be quickly established by checking that the graphs in  Figure~\ref{fig:faux} are cospectral on four vertices for the signless Laplacian. Since these are cospectral for $L_1$, we can now apply Corollary~\ref{cor:emptyglue} for the graphs in Figure~\ref{fig:faux} and conclude that we can attach \emph{any} rooted graph and the result will still be cospectral. Now, using a rooted tree on $k$ vertices we produce a pair of cospectral graphs on $4k$ vertices one of which is a tree and one a non-tree. Actually, we have established a stronger result than existence. Namely, that since the number of trees grows exponentially, then the number of these cospectral tree/non-tree pairs for the signless Laplacian will also grow exponentially.

\begin{figure}[!htb]
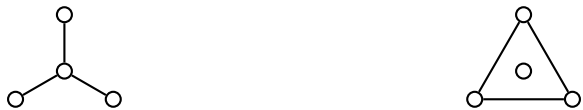

\centering
\FIGSIGNLESSCOSPEC
\caption{A pair of graphs which are cospectral for the signless Laplacian on four vertices.}
\label{fig:faux}
\end{figure}

In terms of this paper we also have the following related result.

\begin{corollary}
If $(T,B_1)$ and $(G,B_2)$ are coalescing cospectral for the signless Laplacian where $T$ is a tree and $G$ is a non-tree, then $|B_1|=|B_2|=4\ell$ for some $\ell$.
\end{corollary}
\begin{proof}
Glue in $G=K_2$ into both $(T,B_1)$ and $(G,B_2)$ and produce another pair of graphs which are cospectral for the signless Laplacian for both a tree and non-tree. 
Now applying Theorem~\ref{thm:4k} for the graphs in our assumption and the newly formed graphs, we have that $|T|=4k$ and $|T|+|B_1|=4j$. The result now follows.
\end{proof}

Figure~\ref{fig:1_example} which shows two non-trees with coalescing sets having sizes not a multiple of $4$, and similarly there are many examples of two trees with coalescing sets having sizes not a multiple of $4$.

\subsection*{Coalescing cospectral sets and graph structure}
One direction of exploration is finding combinatorial properties to help identify coalescing cospectral sets. As an example, suppose that $(H_1,B_1)$ and $(H_2,B_2)$ are coalescing cospectral for the \emph{adjacency} matrix. Then we have that $p_{H_1,B_1}(x)=p_{H_2,B_2}(x)$, which means that the submatrices when removing the rows/columns corresponding to $B_1$ and $B_2$ respectively are cospectral. For the adjacency matrix these submatrices also represent the adjacency matrix for subgraphs, so we can conclude that the induced subgraphs on the vertices of $V(H_1)\setminus B_1$ and $V(H_2)\setminus B_2$ are cospectral. Applying the same argument on the complements via Theorem~\ref{thm:main} we have a similar relationship for the induced subgraphs on the vertices of $B_1$ and $B_2$. We summarize this in the following.

\begin{proposition}
Let $G[U]$ denote the induced subgraph of $G$ on the vertices of $U$. If $(H_1,B_1)$ and $(H_2,B_2)$ are coalescing cospectral for the adjacency matrix ($L_0$), then $H_1[B_1]$ and $H_2[B_2]$ are cospectral as is also $H_1[V(H_1)\setminus B_1]$ and $H_2[V(H_2)\setminus B_2]$.
\end{proposition}

The preceding argument demonstrates that the corresponding submatrices are cospectral; however for $q\ne 0$ these submatrices are not the same as subgraphs and so the corresponding induced subgraphs need not be cospectral (and in many cases have significantly different structure).

The coalescing cospectral sets are not only tied to graph structure but are also sensitive to the choice of matrix $L_q$. It is possible for a pair of graphs to be cospectral for multiple values of $q$ where the coalescing cospectral sets are different depending on the choice of $q$. An example of this is shown in  Figure~\ref{fig:different_glue}.

\begin{figure}[!htb]
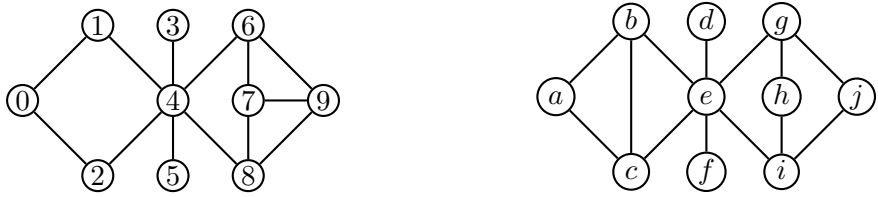

\centering
\FIGDIFFERENTGLUE
\caption{For these two graphs $4{\,:\,}e$ is coalescing cospectral for the Laplacian matrix, but not the adjacency matrix; while $7{\,:\,}b$ is coalescing cospectral for the adjacency matrix, but not the Laplacian matrix.}
\label{fig:different_glue}
\end{figure}

Much remains unknown about tying the coalescing cospectral sets to combinatorial properties of the graphs (either necessary or sufficient conditions).

\subsection*{Other matrices}
Our results have been focused on studying matrices which can be expressed in the form $qD+A$ for some fixed $q$. There are of course many other possible matrices that could be considered. One popular variation is the \emph{normalized adjacency matrix}, $D^{-1/2}AD^{-1/2}$ (which is spectrally equivalent to the probability transition matrix and also a spectral variation of the normalized Laplacian matrix). 

For the normalized adjacency matrix, there is no equivalent result of Theorem~\ref{thm:main}. To see this we consider the graphs shown in Figure~\ref{fig:fail}. On the left are two complete bipartite graphs which are cospectral for the normalized adjacency matrix \cite{chung}. On the right are the two graphs which result from coalescing an edge at each vertex of the graph which are \emph{not} cospectral. In particular, this fails the statement of Corollary~\ref{cor:emptyglue} which means that Theorem~\ref{thm:main} must not hold for the normalized adjacency matrix. The main issue is that while there is a formula for the characteristic polynomial when coalescing, it is more involved and has additional constraints for maintaining cospectrality (see \cite{Guo}).

\begin{figure}[!htb]
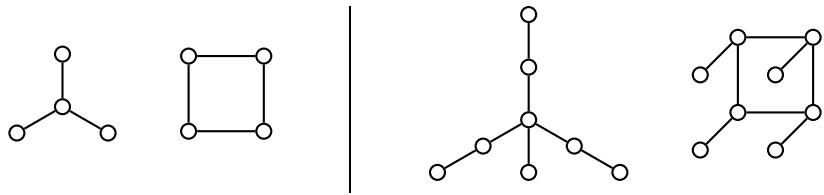

\centering
\FIGFAIL
\caption{The two graphs on the left are cospectral for the normalized adjacency matrix, while the two graphs on the left resulting from coalescing an edge at every vertex are \emph{not} cospectral for the normalized adjacency matrix.}
\label{fig:fail}
\end{figure}

Another matrix which has been studied is the \emph{distance matrix}, $\mathcal{D}$ where $\mathcal{D}_{u,v}$ records the distance between the vertices $u$ and $v$. Through some experimentation (namely for small graphs testing coalescing through a collection of random graphs) it appears that an equivalent result of Theorem~\ref{thm:main} might hold, and we offer the following.

\begin{conjecture}
If  $(H_1,B_1)$ and $(H_2,B_2)$ are coalescing cospectral with respect to $\mathcal{D}$, then $(H_1,V(H_1)\setminus B_1)$ and $(H_2,V(H_2)\setminus B_2)$ are also coalescing cospectral with respect to $\mathcal{D}$.
\end{conjecture}

One obstacle to proving this conjecture is that our techniques hinges on cycle decompositions which allow for a way to express the characteristic polynomial of the result of coalescing two general graphs in simple terms, e.g.\ \eqref{eq:general_charpoly}. However, for the distance matrix the authors are not aware of \emph{any} generic formula for the distance matrix, and the direct generalization of \eqref{eq:general_charpoly} does not hold. (We remark in passing that some special cases of coalescing have been done by Heysse \cite{Heysse} using eigenvector arguments; coalescing involving some restricted trees or the cycle graph has also been done (see \cite{distancesurvey})).

So an important first step in tackling the conjecture might be to find some appropriate generalization of \eqref{eq:general_charpoly}. The authors look forward to seeing more progress in this direction.

\subsection*{Acknowledgments}
This research was conducted primarily at the 2022 Iowa State University Math REU which was supported through NSF Grant DMS-1950583.

\bibliographystyle{plain}
\bibliography{bibliography}
\end{document}

%% file: pictures.tex
\def\FIGFAIL{\begin{tabular}{cc@{\qquad}|@{\qquad}cc}
\begin{tabular}{c}
\begin{tikzpicture}[scale=0.7]

    \pgfmathsetmacro{\r}{1};

    \node[shape=circle,draw=black,inner sep =2pt, thick] (A) at (0,0) {};
    \node[shape=circle,draw=black,inner sep =2pt, thick] (B) at (90:\r) {};
    \node[shape=circle,draw=black,inner sep =2pt, thick] (C) at (210:\r) {};
    \node[shape=circle,draw=black,inner sep =2pt, thick] (D) at (330:\r) {};

    \draw [thick] (A) edge (B);
    \draw [thick](A) edge (C);
    \draw [thick](A) edge (D);
    
\end{tikzpicture}
\end{tabular}
&
\begin{tabular}{c}
\begin{tikzpicture}[scale = 1]

    \node[shape=circle,draw=black,inner sep =2pt, thick] (A) at (0,0) {};
    \node[shape=circle,draw=black,inner sep =2pt, thick] (B) at (0, 1) {};
    \node[shape=circle,draw=black,inner sep =2pt, thick] (C) at (1, 0) {};
    \node[shape=circle,draw=black,inner sep =2pt, thick] (D) at (1, 1) {};

    \draw [thick](B) edge (D);
    \draw [thick](C) edge (D);
    \draw [thick](C) edge (A);
    \draw [thick] (A) edge (B);
    
\end{tikzpicture}
\end{tabular}
&
\begin{tabular}{c}
\begin{tikzpicture}[scale=0.7]

    \pgfmathsetmacro{\r}{1};
    \pgfmathsetmacro{\rr}{2};
    \pgfmathsetmacro{\s}{0.5};

    \node[shape=circle,draw=black,inner sep =2pt, thick] (D) at (0,0) {};
    \node[shape=circle,draw=black,inner sep =2pt, thick] (A) at (90:\r) {};
    \node[shape=circle,draw=black,inner sep =2pt, thick] (B) at (210:\r) {};
    \node[shape=circle,draw=black,inner sep =2pt, thick] (C) at (330:\r) {};
    
    \node[shape=circle,draw=black,inner sep =2pt, thick] (A1) at (0,\rr) {};
    
    \begin{scope}[rotate=120]
        \node[shape=circle,draw=black,inner sep =2pt, thick] (B1) at (0,\rr) {};

    \end{scope}
    
    \begin{scope}[rotate=240]
        \node[shape=circle,draw=black,inner sep =2pt, thick] (C1) at (0,\rr) {};

    \end{scope}
    
    \node[shape=circle,draw=black,inner sep =2pt, thick] (D1) at (0,-1) {};

    \draw [thick] (D) edge (A);
    \draw [thick](D) edge (B);
    \draw [thick](D) edge (C);
    \draw [thick] (A) edge (A1);
    \draw [thick] (B) edge (B1);
    \draw [thick] (C) edge (C1);
    \draw [thick] (D) edge (D1);
    
\end{tikzpicture}
\end{tabular}
&
\begin{tabular}{c}
\begin{tikzpicture}[scale = 1]

    \node[shape=circle,draw=black,inner sep =2pt, thick] (A) at (0,0) {};
    \node[shape=circle,draw=black,inner sep =2pt, thick] (B) at (0, 1) {};
    \node[shape=circle,draw=black,inner sep =2pt, thick] (C) at (1, 0) {};
    \node[shape=circle,draw=black,inner sep =2pt, thick] (D) at (1, 1) {};
    
    \node[shape=circle,draw=black,inner sep =2pt, thick] (A1) at (-.5, -.5) {};
    \node[shape=circle,draw=black,inner sep =2pt, thick] (B1) at (-.5, .5) {};
    \node[shape=circle,draw=black,inner sep =2pt, thick] (C1) at (.5, -.5) {};
    \node[shape=circle,draw=black,inner sep =2pt, thick] (D1) at (.5, .5) {};

    \draw [thick](B) edge (D);
    \draw [thick](C) edge (D);
    \draw [thick](C) edge (A);
    \draw [thick] (A) edge (B);
    
    \draw [thick](A1) edge (A);
    \draw [thick](B1) edge (B);
    \draw [thick](C1) edge (C);
    \draw [thick](D1) edge (D);
    
\end{tikzpicture}
\end{tabular}
\end{tabular}}

\def\FIGDIFFERENTGLUE{\begin{tikzpicture}[scale=0.5]
    \node[inner sep=1pt,thick,shape=circle,draw=black] (A) at (0,2) {\small$0$};
    \node[inner sep=1pt,thick,shape=circle,draw=black] (B) at (2,4) {\small$1$};
    \node[inner sep=1pt,thick,shape=circle,draw=black] (C) at (2,0) {\small$2$};
    \node[inner sep=1pt,thick,shape=circle,draw=black] (D) at (4,2) {\small$4$};

    \draw [thick](A) --(B);
    \draw [thick](A) --(C);
    \draw [thick](C) --(D);
    \draw [thick](B) --(D);
    \node[inner sep=1pt,thick,shape=circle,draw=black] (E) at (4,0) {\small$5$};
    \node[inner sep=1pt,thick,shape=circle,draw=black] (F) at (4,4) {\small$3$};
    
    \draw [thick](E) --(D);
    \draw [thick](F) --(D);
    \node[inner sep=1pt,thick,shape=circle,draw=black] (B1) at (6,0) {\small$8$};
    \node[inner sep=1pt,thick,shape=circle,draw=black] (C1) at (6,4) {\small$6$};
    \node[inner sep=1pt,thick,shape=circle,draw=black] (G) at (6,2) {\small$7$};
    \node[inner sep=1pt,thick,shape=circle,draw=black] (D1) at (8,2) {\small$9$};

    \draw [thick](D) --(B1);
    \draw [thick](D) --(C1);
    \draw [thick](C1) --(D1);
    \draw [thick](B1) --(D1);
    \draw [thick](C1) --(G);
    \draw [thick](B1) --(G);
    \draw [thick](D1) --(G);
\end{tikzpicture}
\hfil
\begin{tikzpicture}[scale=0.5
]
    \node[inner sep=1pt,thick,shape=circle,draw=black] (A) at (0,2) {\small$a\vphantom{f}$};
    \node[inner sep=1pt,thick,shape=circle,draw=black] (B) at (2,4) {\small$b\vphantom{f}$};
    \node[inner sep=1pt,thick,shape=circle,draw=black] (C) at (2,0) {\small$c\vphantom{f}$};
    \node[inner sep=1pt,thick,shape=circle,draw=black] (D) at (4,2) {\small$e\vphantom{f}$};

    \draw [thick](A) --(B);
    \draw [thick](A) --(C);
    \draw [thick](C) --(D);
    \draw [thick](B) --(D);
    \draw [thick](B) --(C);
    \node[inner sep=1pt,thick,shape=circle,draw=black] (E) at (4,0) {\small$f$};
    \node[inner sep=1pt,thick,shape=circle,draw=black] (F) at (4,4) {\small$d\vphantom{f}$};
    
    \draw [thick](E) --(D);
    \draw [thick](F) --(D);
    \node[inner sep=1pt,thick,shape=circle,draw=black] (B1) at (6,0) {\small$i\vphantom{f}$};
    \node[inner sep=1pt,thick,shape=circle,draw=black] (C1) at (6,4) {\small$g\vphantom{f}$};
    \node[inner sep=1pt,thick,shape=circle,draw=black] (G) at (6,2) {\small$h\vphantom{f}$};
    \node[inner sep=1pt,thick,shape=circle,draw=black] (D1) at (8,2) {\small$j\vphantom{f}$};

    \draw [thick](D) --(B1);
    \draw [thick](D) --(C1);
    \draw [thick](C1) --(D1);
    \draw [thick](B1) --(D1);
    \draw [thick](C1) --(G);
    \draw [thick](B1) --(G);
\end{tikzpicture}}

\def\FIGNOUNIONS{\begin{tikzpicture}[scale=0.6]
    \node[inner sep=1pt,thick,shape=circle,draw=black] (A) at (0,0) {\small$2$};
    \node[inner sep=1pt,thick,shape=circle,draw=black] (B) at (0,2) {\small$1$};
    \node[inner sep=1pt,thick,shape=circle,draw=black] (C) at (2,2) {\small$4$};
    \node[inner sep=1pt,thick,shape=circle,draw=black] (D) at (2,0) {\small$5$};
    
    \draw [thick](A)--(B);
    \draw [thick](B)--(C);
    \draw [thick](C)--(D);
    \draw [thick](A)--(D);
    \node[inner sep=1pt,thick,shape=circle,draw=black] (E) at (0,4) {\small$0$};
    \node[inner sep=1pt,thick,shape=circle,draw=black] (F) at (2,4) {\small$3$};
    
    \draw [thick](E)--(F);
    \draw [thick](C)--(F);
    \draw [thick](C)--(E);
    \draw [thick] (B)--(E);
    \draw [line width=4pt, white](B)--(F);
    \draw [thick] (B)--(F);
    \node[inner sep=1pt,thick,shape=circle,draw=black] (G) at (4,2) {\small$6$};
    \node[inner sep=1pt,thick,shape=circle,draw=black] (H) at (6,2) {\small$7$};
    
    \draw [thick](G)--(F);
    \draw [thick](G)--(C);
    \draw [thick](G)--(D);
    \draw [thick](H)--(G);
\end{tikzpicture}\hfil
\begin{tikzpicture}[scale=0.6]
    \node[inner sep=1pt,thick,shape=circle,draw=black] (A) at (0,0) {\small$b\vphantom{f}$};
    \node[inner sep=1pt,thick,shape=circle,draw=black] (B) at (0,2) {\small$a\vphantom{f}$};
    \node[inner sep=1pt,thick,shape=circle,draw=black] (C) at (2,2) {\small$e\vphantom{f}$};
    \node[inner sep=1pt,thick,shape=circle,draw=black] (D) at (2,0) {\small$f$};
    \node[inner sep=1pt,thick,shape=circle,draw=black] (E) at (1,1) {\small$c\vphantom{f}$};
    
    \draw [thick](A)--(B);
    \draw [thick](B)--(C);
    \draw [thick](C)--(D);
    \draw [thick](A)--(D);
    \draw [thick](A)--(E);
    \draw [thick](B)--(E);
    \draw [thick](C)--(E);
    \draw [thick](D)--(E);
    \node[inner sep=1pt,thick,shape=circle,draw=black] (F) at (4,2) {\small$g\vphantom{f}$};
    \node[inner sep=1pt,thick,shape=circle,draw=black] (G) at (6,2) {\small$h\vphantom{f}$};
    
    \draw [thick](D)--(F);
    \draw [thick](C)--(F);
    \draw [thick](G)--(F);
    \node[inner sep=1pt,thick,shape=circle,draw=black] (H) at (2,4) {\small$d\vphantom{f}$};
    
    \draw [thick](B)--(H);
    \draw [thick](C)--(H);
\end{tikzpicture}
}

\def\FIGHALFEXAMPLE{
\begin{tikzpicture}[scale=0.6]
    \node[inner sep=1pt,thick,shape=circle,draw=black] (A) at (0,0) {\small$3$};
    \node[inner sep=1pt,thick,shape=circle,draw=black] (B) at (2,0) {\small$4$};
    \node[inner sep=1pt,thick,shape=circle,draw=black] (C) at (1,{sqrt(3)}) {\small$1$};
    
    \draw [thick](B)--(A);
    \draw [thick](C)--(A);
    \draw [thick](B)--(C);
    
    \node[inner sep=1pt,thick,shape=circle,draw=black] (D) at (4,0) {\small$5$};
    \node[inner sep=1pt,thick,shape=circle,draw=black] (E) at (3,{sqrt(3)}) {\small$2$};
    
    \draw [thick](B)--(D);
    \draw [thick](B)--(E);
    \draw [thick](D)--(E);
    
    \node[inner sep=1pt,thick,shape=circle,draw=black] (F) at (2,{2*sqrt(3)}) {\small$0$};
    \node[inner sep=1pt,thick,shape=circle,draw=black] (G) at (2,{2*sqrt(1/3)}){\small$6$};
    
    \draw [thick](C)--(E);
    \draw [thick](C)--(F);
    \draw [thick](E)--(F);
\end{tikzpicture}
\hfil
\begin{tikzpicture}[scale=0.6]
    \node[inner sep=1pt,thick,shape=circle,draw=black] (B) at (1,0) {\small$f$};
    \node[inner sep=1pt,thick,shape=circle,draw=black] (C) at (0,{sqrt(3)}) {\small$c\vphantom{f}$};
    \node[inner sep=1pt,thick,shape=circle,draw=black] (A) at (2,{sqrt(3)}) {\small$d\vphantom{f}$};
    
    \draw [thick](A)--(B);
    \draw [thick](A)--(C);
    \draw [thick](B)--(C);
    
    \node[inner sep=1pt,thick,shape=circle,draw=black] (B2) at (3,0) {\small$g\vphantom{f}$};
    \node[inner sep=1pt,thick,shape=circle,draw=black] (C2) at (4,{sqrt(3)}) {\small$e\vphantom{f}$};
    
    \draw [thick](A)--(B2);
    \draw [thick](A)--(C2);
    \draw [thick](B2)--(C2);
    
    \node[inner sep=1pt,thick,shape=circle,draw=black] (B3) at (1,{2*sqrt(3)}) {\small$a\vphantom{f}$};
    \node[inner sep=1pt,thick,shape=circle,draw=black] (C3) at (3,{2*sqrt(3)}) {\small$b\vphantom{f}$};
    
    \draw [thick](A)--(B3);
    \draw [thick](A)--(C3);
    \draw [thick](B3)--(C3);
\end{tikzpicture}
}

\def\FIGADJACENCY{
\begin{tikzpicture}[scale=0.6]
    \node[inner sep=1pt,thick,shape=circle,draw=black] (A) at (0,0) {\small$3$};
    \node[inner sep=1pt,thick,shape=circle,draw=black] (B) at (2,0) {\small$4$};
    \node[inner sep=1pt,thick,shape=circle,draw=black] (C) at (1,{sqrt(3)}) {\small$1$};
    
    \draw [thick](B)--(A);
    \draw [thick](C)--(A);
    \draw [thick](B)--(C);
    
    \node[inner sep=1pt,thick,shape=circle,draw=black] (D) at (4,0) {\small$5$};
    \node[inner sep=1pt,thick,shape=circle,draw=black] (E) at (3,{sqrt(3)}) {\small$2$};
    
    \draw [thick](B)--(D);
    \draw [thick](B)--(E);
    \draw [thick](D)--(E);
    
    \node[inner sep=1pt,thick,shape=circle,draw=black] (F) at (2,{2*sqrt(3)}) {\small$0$};
    \node[inner sep=1pt,thick,shape=circle,draw=black] (G) at (2,{2*sqrt(1/3)}){\small$6$};
    
    \draw [thick](C)--(E);
    \draw [thick](C)--(F);
    \draw [thick](E)--(F);
\end{tikzpicture}
\hfil
\begin{tikzpicture}[scale=0.4]
    \node[inner sep=1pt,thick,shape=circle,draw=black] (D) at (0,0) {\small$c\vphantom{f}$};
    \node[inner sep=1pt,thick,shape=circle,draw=black] (A) at (90:2) {\small$b\vphantom{f}$};
    \node[inner sep=1pt,thick,shape=circle,draw=black] (B) at (90+120:2) {\small$d\vphantom{f}$};
    \node[inner sep=1pt,thick,shape=circle,draw=black] (C) at (90+240:2) {\small$e\vphantom{f}$};
    
    \draw [thick](B)--(A);
    \draw [thick](C)--(A);
    \draw [thick](B)--(C);
    \draw [thick](A)--(D);
    \draw [thick](B)--(D);
    \draw [thick](C)--(D);
    
    \node[inner sep=1pt,thick,shape=circle,draw=black] (E) at (90:4) {\small$a\vphantom{f}$};
    \node[inner sep=1pt,thick,shape=circle,draw=black] (F) at (90+120:4) {\small$f$};
    \node[inner sep=1pt,thick,shape=circle,draw=black] (G) at (90+240:4) {\small$g\vphantom{f}$};
    
    \draw [thick](A)--(E);
    \draw [thick](B)--(F);
    \draw [thick](C)--(G);
\end{tikzpicture}}

\def\FIGLAPLACIAN{
\begin{tikzpicture}[scale=0.5]
    \node[inner sep=1pt,thick,shape=circle,draw=black] (A) at (0+36:2) {\small$3$};
    \node[inner sep=1pt,thick,shape=circle,draw=black] (B) at (72+36:2) {\small$2$};
    \node[inner sep=1pt,thick,shape=circle,draw=black] (C) at (72*2+36:2) {\small$1$};
    \node[inner sep=1pt,thick,shape=circle,draw=black] (D) at (72*3+36:2) {\small$5$};
    \node[inner sep=1pt,thick,shape=circle,draw=black] (E) at (72*4+36:2) {\small$6$};
    \node[inner sep=1pt,thick,shape=circle,draw=black] (G) at ({1.61803398575+2.03614784182},0) {\small$4$};
    \node[inner sep=1pt,thick,shape=circle,draw=black] (F) at (-4,0) {\small$0$};
    
    \draw [thick](A)--(B);
    \draw [thick](C)--(B);
    \draw [thick](C)--(D);
    \draw [thick](D)--(E);
    \draw [thick](A)--(E);
    \draw [thick](C)--(F);

    \draw [thick](A)--(G);
    \draw [thick](E)--(G);
\end{tikzpicture}
\hfil
\begin{tikzpicture}[scale=0.5]
    \node[inner sep=1pt,thick,shape=circle,draw=black] (A) at (0+36:2) {\small$c\vphantom{f}$};
    \node[inner sep=1pt,thick,shape=circle,draw=black] (B) at (72+36:2) {\small$b\vphantom{f}$};
    \node[inner sep=1pt,thick,shape=circle,draw=black] (C) at (72*2+36:2) {\small$a\vphantom{f}$};
    \node[inner sep=1pt,thick,shape=circle,draw=black] (D) at (72*3+36:2) {\small$f$};
    \node[inner sep=1pt,thick,shape=circle,draw=black] (E) at (72*4+36:2) {\small$g\vphantom{f}$};
    \node[inner sep=1pt,thick,shape=circle,draw=black] (G) at ({1.61803398575+2.03614784182},0) {\small$d\vphantom{f}$};
    \node[inner sep=1pt,thick,shape=circle,draw=black] (F) at ({1.61803398575+2.03614784182+2},0) {\small$e\vphantom{f}$};
    
    \draw [thick](A)--(B);
    \draw [thick](C)--(B);
    \draw [thick](C)--(D);
    \draw [thick](D)--(E);
    \draw [thick](A)--(E);
    \draw [thick](A)--(G);
    \draw [thick](E)--(G);
    \draw [thick](F)--(G);
\end{tikzpicture}}

\def\FIGSIGNLESS{
\begin{tikzpicture}[scale=0.7]
    \node[inner sep=1pt,thick,shape=circle,draw=black] (A) at (0,0) {\small$4$};
    \node[inner sep=1pt,thick,shape=circle,draw=black] (B) at (0,2) {\small$0$};
    \node[inner sep=1pt,thick,shape=circle,draw=black] (C) at (2,2) {\small$1$};
    \node[inner sep=1pt,thick,shape=circle,draw=black] (D) at (2,0) {\small$5$};
    
    \draw [thick](A)--(B);
    \draw [thick](B)--(C);
    \draw [thick](D)--(C);
    \draw [thick](A)--(D);
    \draw [thick](A)--(C);
    
    \node[inner sep=1pt,thick,shape=circle,draw=black] (E) at (4,0) {\small$6$};
    \node[inner sep=1pt,thick,shape=circle,draw=black] (F) at (4,2) {\small$2$};
    
    \draw [thick](C)--(E);
    \draw [thick](C)--(F);
    \draw [thick](D)--(E);
    \draw [line width=4pt,color=white](D)--(F);

    \draw [thick](D)--(F);
    
    \node[inner sep=1pt,thick,shape=circle,draw=black] (G) at ({4+(sqrt(3))},1) {\small$3$};
    \draw [thick](E)--(G);
    \draw [thick](F)--(G);
    
\end{tikzpicture}
\hfil
\begin{tikzpicture}[scale=0.7]
    \node[inner sep=1pt,thick,shape=circle,draw=black] (A) at (0,0) {\small$e\vphantom{f}$};
    \node[inner sep=1pt,thick,shape=circle,draw=black] (B) at (0,2) {\small$a\vphantom{f}$};
    \node[inner sep=1pt,thick,shape=circle,draw=black] (C) at (2,2) {\small$b\vphantom{f}$};
    \node[inner sep=1pt,thick,shape=circle,draw=black] (D) at (2,0) {\small$f$};
    
    \draw [thick](A)--(B);
    \draw [thick](B)--(C);
    \draw [thick](D)--(C);
    \draw [thick](A)--(D);
    
    \node[inner sep=1pt,thick,shape=circle,draw=black] (E) at (4,0) {\small$g\vphantom{f}$};
    \node[inner sep=1pt,thick,shape=circle,draw=black] (F) at (4,2) {\small$c\vphantom{f}$};
    \draw [thick](C)--(E);
    \draw [thick](C)--(F);
    \draw [thick](D)--(E);
    \draw [line width=4pt,color=white](D)--(F);
    \draw [thick](D)--(F);
    \draw [thick](E)--(F);
    
    \node[inner sep=1pt,thick,shape=circle,draw=black] (G) at ({4+(sqrt(3))},1) {\small$d\vphantom{f}$};
    \draw [thick](E)--(G);
    \draw [thick](F)--(G);
\end{tikzpicture}}

\def\FIGSIGNLESSCOSPEC{\begin{tikzpicture}[scale=0.75]
    \node[draw,circle,thick,inner sep=2pt] (a) at (0,0) {};
    \node[draw,circle,thick,inner sep=2pt] (b) at (90:1) {};
    \node[draw,circle,thick,inner sep=2pt] (c) at (-30:1) {};
    \node[draw,circle,thick,inner sep=2pt] (d) at (210:1) {};
    \draw[thick] (a)--(b) (a)--(c) (a)--(d);
\end{tikzpicture}
\hfil
\begin{tikzpicture}[scale=0.75]
    \node[draw,circle,thick,inner sep=2pt] (a) at (0,0) {};
    \node[draw,circle,thick,inner sep=2pt] (b) at (90:1) {};
    \node[draw,circle,thick,inner sep=2pt] (c) at (-30:1) {};
    \node[draw,circle,thick,inner sep=2pt] (d) at (210:1) {};
    \draw[thick] (b)--(c)--(d)--(b);
\end{tikzpicture}}

\def\FIGNOTWO{\begin{tabular}{c@{\quad}|@{\quad}c}
\begin{tabular}{c}
\begin{tikzpicture}[scale=0.8]
\node[draw,thick,circle,inner sep=1pt]
    (0) at ( 90:1) {\small$f\vphantom{f}$};
\node[draw,thick,circle,inner sep=1pt]
    (1) at ( 18:1) {\small$g\vphantom{f}$};
\node[draw,thick,circle,inner sep=1pt]
    (2) at (162:1) {\small$e\vphantom{f}$};
\node[draw,thick,circle,inner sep=1pt]
    (3) at (-54:1) {\small$h\vphantom{f}$};
\node[draw,thick,circle,inner sep=1pt]
    (4) at (234:1) {\small$d\vphantom{f}$};
\node[draw,thick,circle,inner sep=1pt]
    (5) at ( 90:2) {\small$b\vphantom{f}$};
\node[draw,thick,circle,inner sep=1pt]
    (6) at ( 18:2) {\small$c\vphantom{f}$};
\node[draw,thick,circle,inner sep=1pt]
    (7) at (162:2) {\small$a\vphantom{f}$};
\foreach \x/\y in {0/1,0/2,0/3,0/4,0/5,1/2,1/3,1/4,1/6,2/3,2/4,2/7}{
    \draw[line width=4pt, white] (\x)--(\y);
    \draw[thick] (\x)--(\y);
}
\end{tikzpicture} \quad
\begin{tikzpicture}[scale=1]
\node[draw,thick,circle,inner sep=1pt]
    (0) at (  0:0) {\small$0$};
\node[draw,thick,circle,inner sep=1pt]
    (1) at ( 30:1) {\small$1$};
\node[draw,thick,circle,inner sep=1pt]
    (2) at ( 90:1) {\small$2$};
\node[draw,thick,circle,inner sep=1pt]
    (3) at (150:1) {\small$3$};
\node[draw,thick,circle,inner sep=1pt]
    (4) at (210:1) {\small$4$};
\node[draw,thick,circle,inner sep=1pt]
    (5) at (270:1) {\small$5$};
\node[draw,thick,circle,inner sep=1pt]
    (6) at (330:1) {\small$6$};
\node[draw,thick,circle,inner sep=1pt]
    (7) at (  0:1.4) {\small$7$};
\foreach \x/\y in {2/4,2/6,4/6,0/1,0/2,0/3,0/4,0/5,0/6,1/6,2/3,4/5}{
    \draw[line width=4pt, white] (\x)--(\y);
    \draw[thick] (\x)--(\y);
}
\end{tikzpicture}
\end{tabular}&
\begin{tabular}{c}
\begin{tikzpicture}[scale=0.8]
\node[draw,thick,circle,inner sep=2pt]
    (0) at ( 90:1) {};
\node[draw,thick,circle,inner sep=2pt]
    (1) at ( 18:1) {};
\node[draw,thick,circle,inner sep=2pt]
    (2) at (162:1) {};
\node[draw,thick,circle,inner sep=2pt]
    (3) at (-54:1) {};
\node[draw,thick,circle,inner sep=2pt]
    (4) at (234:1) {};
\node[draw,thick,circle,inner sep=2pt]
    (5) at ( 90:2) {};
\node[draw,thick,circle,inner sep=2pt]
    (6) at ( 18:2) {};
\node[draw,thick,circle,inner sep=2pt]
    (7) at (162:2) {};
\node[draw,thick,circle,inner sep=2pt]
    (8) at (0.95-0.25,0.31+0.7) {};
\node[draw,thick,circle,inner sep=2pt]
    (9) at (0.95+0.25,0.31+0.7) {};
\node[draw,thick,circle,inner sep=2pt]
    (10) at (-0.5,1.5) {};
\node[draw,thick,circle,inner sep=2pt]
    (11) at (0.7,2.25) {};
\node[draw,thick,circle,inner sep=2pt]
    (12) at (0.7,1.75) {};
\node[draw,thick,circle,inner sep=2pt]
    (13) at (-1.9,1.36) {};

\foreach \x/\y in {0/1,0/2,0/3,0/4,0/5,1/2,1/3,1/4,1/6,2/3,2/4,2/7,7/13,0/10,5/11,5/12,1/8,1/9}{
    \draw[line width=4pt, white] (\x)--(\y);
    \draw[thick] (\x)--(\y);
}
\end{tikzpicture} \quad
\begin{tikzpicture}[scale=1]
\node[draw,thick,circle,inner sep=2pt]
    (0) at (  0:0) {};
\node[draw,thick,circle,inner sep=2pt]
    (1) at ( 30:1) {};
\node[draw,thick,circle,inner sep=2pt]
    (2) at ( 90:1) {};
\node[draw,thick,circle,inner sep=2pt]
    (3) at (150:1) {};
\node[draw,thick,circle,inner sep=2pt]
    (4) at (210:1) {};
\node[draw,thick,circle,inner sep=2pt]
    (5) at (270:1) {};
\node[draw,thick,circle,inner sep=2pt]
    (6) at (330:1) {};
\node[draw,thick,circle,inner sep=2pt]
    (7) at (  0:1.4) {};
\node[draw,thick,circle,inner sep=2pt]
    (8) at (  .87,1) {};
\node[draw,thick,circle,inner sep=2pt]
    (9) at (  0,1.5) {};
\node[draw,thick,circle,inner sep=2pt]
    (10) at (  -.87,1) {};
\node[draw,thick,circle,inner sep=2pt]
    (11) at (  -1.25,1) {};
\node[draw,thick,circle,inner sep=2pt]
    (12) at (  -.87,0) {};
\node[draw,thick,circle,inner sep=2pt]
    (13) at (  -1.25,0) {};
\foreach \x/\y in {2/4,2/6,4/6,0/1,0/2,0/3,0/4,0/5,0/6,1/6,2/3,4/5,1/8,2/9,3/10,3/11,4/12,4/13}{
    \draw[line width=4pt, white] (\x)--(\y);
    \draw[thick] (\x)--(\y);
}
\end{tikzpicture}
\end{tabular}
\end{tabular}
}